\documentclass[a4paper,12pt]{amsart}
\usepackage{amsthm,amssymb}
\usepackage[T1]{fontenc}
\usepackage[utf8]{inputenc}
\usepackage{enumitem}
\usepackage{mathrsfs}
\usepackage{mathtools}
\usepackage{enumitem}
\usepackage[margin=3.25cm]{geometry}
\usepackage[numbers, sort&compress]{natbib}
\usepackage[all]{xy}
\usepackage{xspace}
\usepackage{caption}
\numberwithin{equation}{section}

\usepackage[hypertexnames=false]{hyperref}
\hypersetup{
    colorlinks,
    linkcolor={red!90!black},
    citecolor={green!70!black},
    urlcolor={blue!80!black}
}

\usepackage{tikz}

\newtheorem{theorem}{Theorem}[section]
\newtheorem{lemma}[theorem]{Lemma}

\newtheorem{corollary}[theorem]{Corollary}

\newtheoremstyle{mytheoremstyle} 
    {1em plus .2em minus .1em}                    
    {1em plus .2em minus .1em}                    
    {\rmfamily}                   
    {}                           
    {\bfseries}                   
    {.}                          
    {.5em}                       
    {}  

\theoremstyle{mytheoremstyle}

\newtheorem{example}[theorem]{Example}

\newcommand{\rarrow}{\rightarrow}
\newcommand{\QED}{\hfill$\dashv$}
\newcommand{\zero}{\mathbf{0}} 
\newcommand{\one}{\mathbf{1}} 
\newcommand\defeq{\coloneqq} 
\newcommand\iffdef{\;\mathrel{\mathord{:}\mathord{\longleftrightarrow}}\;}
\newcommand\defeqtt{\mathtt{df}\,}

\newcommand\klam[1]{\left\langle#1\right\rangle}
\newcommand\mfr{\mathfrak}

\newcommand{\rsc}{\mathfrak{R^{\mathtt{rsc}}}}
\newcommand{\rs}{\mathfrak{R^{\mathtt{rs}}}}
\newcommand\frB{\mfr{B}}
\newcommand\frR{\mfr{R}}
\DeclareMathOperator{\RC}{RC}
\DeclareMathOperator{\con}{\mathrel{\mathsf{C}}}
\newcommand{\conbar}{\mathrel{\overline{\con}}}
\newcommand{\notconbar}{\mathrel{(-\!\conbar)}}
\newcommand{\notcon}{\mathrel{(-\!\con)}} 
\newcommand{\ConRel}{\ensuremath{\mathfrak{C}}\xspace}
\DeclareMathOperator{\overl}{\mathsf{O}}
\DeclareMathOperator{\Cl}{Cl}
\DeclareMathOperator{\upop}{\uparrow} 

\DeclareMathOperator{\Clan}{\mathsf{Clan}} 
\DeclareMathOperator{\Atom}{At}
\newcommand{\Implies}{\Rightarrow}
\newcommand{\Iff}{\Longleftrightarrow}
\newcommand{\tand}{\text{ and }}
\newcommand{\tor}{\text{ or }}
\newcommand{\qtiff}{\quad\text{iff}\quad}
\newcommand{\timplies}{\text{ implies }}
\newcommand{\set}[1]{\ensuremath{\{#1\}}}
\newcommand{\z}{\emptyset}
\newcommand{\wlg}{w.l.o.g.\xspace }
\newcommand{\Prim}{\mathsf{Prim}}
\newcommand{\PrimI}{\mathsf{PrimI}}
\newcommand{\PrimF}{\mathsf{PrimF}}
\newcommand{\tiff}{\text{\ if and only if\ }\xspace}
\newcommand{\aright}{``$\Rightarrow$'': \ }
\newcommand{\aleft}{``$\Leftarrow$'': \ }
\newcommand{\onto}{\twoheadrightarrow}
\newcommand{\into}{\hookrightarrow}
\newcommand\frD{\mfr{D}}
\newcommand{\ua}[1]{\ensuremath{\mathop{\uparrow}#1}}
\newcommand{\bL}{\ensuremath{\frB(L)}\xspace}
\newcommand{\dL}{\ensuremath{\frD(L)}\xspace}
\newcommand{\sbe}{``$\subseteq$'': \ }
\newcommand{\spe}{``$\supseteq$'': \ }
\newcommand{\da}[1]{\ensuremath{\mathop{\downarrow}#1}}
\newcommand{\Base}{\ensuremath{\mathtt{B}}}
\renewcommand{\S}{\mathcal{S}}
\DeclareMathOperator{\underl}{\mathrel{\mathsf{U}}}
\DeclareMathOperator{\Filt}{\mathsf{Filt}}
\DeclareMathOperator{\Intr}{Int}

\title[Contact lattices and extensions]{Contact relations on bounded distributive lattices and related structures}

\author[]{Ivo D\"untsch and Rafa\l\ Gruszczy\'nski}

\date{}

\address{Ivo D\"untsch\\
Department of Computer Science\\
Brock University\\	
St Ca\-tha\-ri\-nes, Ontario\\
Canada\\
\textsc{Orcid:} 0000-0001-8907-2382}

\email{duentsch@brocku.ca}

\urladdr{https://www.cosc.brocku.ca/~duentsch/}

\address{Rafa\l\ Gruszczy\'nski, \textsc{Orcid:} 0000-0002-3379-0577\\
Department of Logic\\
Institute of Philosophy\\
Nicolaus Copernicus University in Toru\'n\\
Poland}

\email{gruszka@umk.pl}

\begin{document}

\begin{abstract}

We study a discrete representation of contact relations on bounded distributive lattices and some of their reducts. In the course of the paper, we apply ideas and techniques developed by Ivo D\"untsch, Dimiter Vakarelov, and Michael Winter, and we show---among others---that contact relations on a distributive join semi-lattice $L$ with the bottom element correspond to reflexive and symmetric relations which are closed in the product of the Stone space of~$L$.

\smallskip

\noindent Keywords: contact relations, contact algebras, discrete representation

\smallskip

\noindent MSC 2020: Primary 06B15, Secondary 03G10

\end{abstract}

\maketitle

\section{Introduction}\label{sec:intro}

The roots of lattice based algebras with a contact relation can be found in the works of \citet{les_math} on the ``part-of'' relation of mereology, and, on the other hand, in a region-based approach to geometry, where regions instead of points are taken as the basic entity \cite{deLaguna-PLS,Nicod-GITSW,Tarski-FGC,Whitehead-PR}. In this ``pointless geometry'', points are now second order definable as sets of regions, similar to the representation of Boolean algebras, where points can be recovered as ultrafilters. De Laguna's \cite{deLaguna-PLS}, Nicod's \cite{Nicod-GITSW}, and later Whitehead's \citep{Whitehead-PR} addition  to the mereological structures of \citeauthor{les_math} (which arise from  complete Boolean algebras when the smallest element is removed) was a ``connection'' (or ``contact'') relation $\con$ among nonempty regions, which, in its simplest form, is a reflexive and binary relation  that satisfies a compatibility axiom with the order of the algebra as well as distributivity over finite joins.
Based on the work of \citeauthor{les_math} and \citeauthor{Whitehead-PR}, \citet{Clarke-CIBC} presented an axiom system for a ``Calculus of individuals'' founded on an ``in contact with'' relation. The intended domain is such that
\begin{sloppypar}
\begin{quote}
``\ldots we may interpret the individual variables as ranging over
spatial-temporal regions and the two-place primitive `$x$ is connected
with $y$' as a rendering of `$x$ and $y$ share a common point'.''
\cite[p.205]{Clarke-CIBC}.
\end{quote}
\end{sloppypar}
In a parallel development, \emph{proximity structures} have been investigated in a topological context since the 1950s. Proximity spaces are relational structures on families of sets that satisfy axioms which to some extent coincide with those for Clarke's contact structures \citep{Naimpally-et-al-PS}.

Standard models of contact structures are collections of regular open sets of topological spaces
$\klam{X,\tau}$ with the \emph{standard (Whiteheadian) contact} $\con$ among regions, which is defined by
\begin{gather*}
u\con v \iffdef \Cl(u) \cap \Cl(v) \neq \z.
\end{gather*}
Here $u \con v$ denotes the fact that the open sets $u$ and $v$ are in contact, and $\Cl(u)$ and $\Cl(v)$ denote the closure of $u$ and $v$, respectively. Notice that it is possible for open sets to be in contact and having an empty intersection, e.g., two open discs in the Euclidean plane that have overlapping boundaries. Due to isomorphisms between the algebras of regular open and regular closed subsets of $X$, the models built from the latter sets are, so to speak, equivalent, with the contact defined by
\begin{gather*}
u\con v \iffdef u \cap v \neq \z.
\end{gather*}

The abstract version of such standard models are \emph{Boolean contact algebras} (BCAs, for short), that is, Boolean algebras augmented by a contact relation. It was shown that each such algebra has a representation as a subalgebra of the Boolean algebra of regular open (or closed) sets of a topological space with standard contact \cite{Dimov-et-al-CARBTSPA1}. As the notion of ``complement of a region'' was not immediately intuitive, other lattice based algebras with a contact relation were investigated. It turned out, that the topological representation of such structures as collections of regular open (or closed) sets with standard contact was rather limited \cite{Duntsch-et-al-DCLTR}. An alternative representation of Boolean contact algebras by symmetric and reflexive relations on its spectral space was investigated in \cite{Duntsch-et-al-RBTODSAPA} and \cite{dw_cl}, and it is this representation that we shall consider in the present article for structures weaker than Boolean algebras.

The paper is structured as follows. After setting our notation and basic definitions in Section \ref{sec:clan}, we shall introduce contact relations on bounded distributive lattices in Section \ref{sec:contact}. In the following section we review the standard representation for Boolean contact algebras and demonstrate that it will not work unchanged for weaker structures. In Section \ref{sec:discrep} we introduce discrete representations and show that each bounded distributive lattice has such a representation. We show that this result can be generalized to distributive join-semilattices with zero. The structure of the collection of contact relations is exhibited in Section~\ref{sec:conlat}, followed by some closing remarks.

\section{First definitions and notation}\label{sec:clan}

We assume that the reader has a basic knowledge of lattice theory and topology. For structures $S_0,S_1$  of the same type we write $S_0 \leq S_1$, if $S_0$ is a substructure of $S_1$. A binary relation $R$ on a partially orderd set $\klam{U,\leq}$ is called \emph{extensional} if
\begin{gather}\tag{Ext}\label{C5}
R(x) \subseteq R(y) \timplies x \leq y;
\end{gather}
here, $R(x) \defeq \set{y: x \mathrel{R} y}$. The identity relation on $U$ will be denoted by $1'_U$ or simply by $1'$ if $U$ is understood. With some abuse of notation we will usually identify structures with their universe, if no confusion can arise. Throughout,  $\klam{L, + , \cdot, \zero, \one}$ is a non-trivial bounded distributive lattice unless otherwise indicated.  $L^+$ denotes the set of non-zero elements of $L$. If $a \in L$, then $\ua{a} \defeq \set{b: a \leq b}$ is the principal filter generated by $a$.  We denote the set of proper filters of $L$ by $\Filt(L)$ and the set of prime filters by $\Prim(L)$; these sets are partially ordered by $\subseteq$. Maximal elements of $\Filt(L)$ are called \textit{ultrafilters}. By Zorn's Lemma, each filter is contained in a maximal filter which turns out to be prime in a distributive lattice. If $\klam{L,-}$ is a Boolean algebra, then each prime filter is an ultrafilter, consequently, the order on $\Prim(L)$ is the identity relation.

A \emph{p-algebra} $\klam{L,+,\cdot, {}^\ast, \zero, \one}$ is a structure of type $\klam{2,2,1,0,0}$ such that the reduct $\klam{L,+,\cdot, \zero, \one}$ is a bounded distributive lattice\footnote{In the literature a p-algebra is not always distributive.}, and the operation ${}^\ast$ is \emph{pseudocomplementation}, i.e.
\begin{gather}\label{def:*}
a \cdot b = \zero \Iff b \leq a^\ast.
\end{gather}
It is well known, that the class of p-algebras is equational, see e.g. \cite[Chapter VIII]{bd74}, where further properties of p-algebras can be found. The following sets are decisive for a p-algebra:
\begin{xalignat*}{2}
\bL &\defeq \set{a^\ast: a \in L}, &&\text{the \emph{skeleton}}, \\
\dL &\defeq \set{a \in L: a^\ast = \zero}, &&\text{the \emph{set of dense elements}.}
\end{xalignat*}
The following properties of p-algebras are well known \cite[Chapter VIII]{bd74}:
\begin{lemma}\label{lem:propL}
Let $L$ be a p-algebra. Then,
\begin{enumerate}
\item $\klam{\bL, \lor, \land, - , \zero, \one}$ is a Boolean algebra with
\begin{gather*}
a \land b \defeq a \cdot b, \ a \lor b \defeq (a+b)^{\ast\ast}, \tand -a \defeq a^\ast.
\end{gather*}
 The map $r_L\colon L \onto \frB(L)$ defined by $r_L(a) \defeq a^{\ast\ast}$ is a surjective homomorphism and $\frB(L) \cong L/\frD(L)$.
\item $\frD(L)$ is a filter of $L$, and $d \in\frD(L)$ if and only if $a \cdot d \neq \zero$ for all $a \neq \zero$.
\item $a \cdot b = \zero$ \tiff $a^{\ast\ast} \cdot b^{\ast\ast} =\zero$.
\end{enumerate}
\end{lemma}

The following observation has appeared in various places, see \cite{var82}.
\begin{lemma}\label{lem:max}
Let $F$ be a proper filter of the p-algebra $L$. Then, the following are equivalent:
\begin{enumerate}
\item $F$ is maximal.
\item $a \not\in F$ implies $a^\ast \in F$ for all $a \in L$.
\item $F$ is prime, and $a^{\ast\ast} \in F$ implies $a \in F$ for all $a \in L$.
\item $F$ is prime, and $\frD(L) \subseteq F$.
\end{enumerate}
\end{lemma}

Our standard reference for distributive lattices is \cite{bd74}, for Boolean algebras (BAs) \cite{Koppelberg-GTBA}, and for topology \cite{eng77} where all unexplained concepts can be found.

Finally in this section, we consider the topological representation of bounded distributive lattices. Representation theorems for lattice-based structures $S$ often are based on an embedding $h\colon S \to 2^{M}$, where $M$ is a collection of subsets of $S$, for example, ultrafilters in the case of Boolean algebras, prime filters for distributive lattices, and more involved structures in the case of non-distributive lattices. The sets $h(a)$ are taken as a basis for the open sets of a topology. The first such representations for lattice based structures were those by \citet{Stone-TRBA,stone_dl} for bounded distributive lattices and Boolean algebras, of which we recall the former:
\begin{theorem} [{\citealp{stone_dl}}]\label{thm:repDL}
Let $L$ be a bounded distributive lattice, $\Prim(L)$ be the set of its prime filters, and $h\colon L \to 2^{\Prim(L)}$ be defined by $h(a) \defeq \set{F \in \Prim(L): a \in F}$. Then,
\begin{enumerate}
\item $h$ is a lattice embedding into the set algebra $2^{\Prim(L)}$.
\item The collection $\set{h(a): a \in L}$ forms a basis for the open sets of a compact $T_0$-topology $\tau$ on $\Prim(L)$ and each set $h(a)$ is compact open.
\end{enumerate}
\end{theorem}
The topology $\tau$ has other significant properties, which we do  not require in the present context. The topological space $\klam{\Prim(L), \tau}$ is called the \emph{Stone space} or \emph{spectral space} of $L$, denoted by $\S(L)$. We note that in this topology $\Cl(\set{H}) = \da{H}$, unlike the case when the spectral space is obtained from prime ideals.

\section{Contact structures}\label{sec:contact}

A \emph{contact relation} on $L$ is a binary relation $\con$ on $L$ which satisfies the following properties (by means of $\notcon$ we denote the set-theoretical complement of $\con$):
\begin{gather}
    (\forall x\in L)\,\zero\notcon  x\,,\tag{C0}\label{C0}\\
    (\forall x\in L)\,(x\neq\zero\timplies x\con x)\,,\tag{C1}\label{C1}\\
    (\forall x,y\in L)\,(x\con y\timplies y\con x)\,,\tag{C2}\label{C2}\\
    (\forall x,y,z\in L)\,(x\con y\tand y\leq z\timplies x\con z)\,,\tag{C3}\label{C3}\\
    (\forall x,y,z\in L)\,(x\con (y+z) \timplies x\con y\tor x\con z)\,.\tag{C4}\label{C4}
\end{gather}
Below, $\con$ will be the generic name for a contact relation on $L$. A bounded distributive lattice $L$ endowed with a contact relation will be called a \emph{DLC}. \eqref{C3} is a weak connection between the lattice ordering and the contact relation. If $\con$ is extensional, then $\leq$ can be defined by $\con$.
The smallest contact relation on $L$ is the\emph{ overlap relation} defined in the following way:
\begin{equation}\tag{$\defeqtt{\overl}$}\label{df:overl}
x\overl y\iffdef x\cdot y\neq\zero\,.
\end{equation}
The largest contact relation on $L$ is $\con_{\max}\defeq \set{\klam{a,b} \in L^2: a \neq \zero \tand b \neq \zero}$. The collection of contact relations on $L$ is denoted by \ConRel.

Since the axioms \eqref{C0}--\eqref{C4} are universal, the restriction of $\con$ to a bounded sublattice of $L$ is also a contact relation. Conversely, a contact relation on a substructure can be extended, and an extension result was proved in \cite{Duntsch-Winter-CBCA} for Boolean contact algebras and dense subalgebras. It also holds for bounded distributive lattices:

\begin{theorem}\label{thm:extendC}
Let $L,M$ be bounded distributive lattices, $L$ a bounded sublattice of $M$, and $\con$ be a contact relation on $L$. Then, there is a contact relation $\conbar$ on $M$ such that $\klam{L,\con}$ is a substructure of $\klam{M, \conbar}$. Furthermore, $\conbar$ is the largest extension of $\con$ to $M$ for which $L^{2}\cap\mathord{\conbar}=\mathord{\con}$.
\end{theorem}
\begin{proof}
For $a,b \in M$ set
\begin{equation}\tag{$\defeqtt{\conbar}$}\label{df:conbar}
a \conbar b \iffdef (\ua{a} \cap L) \times (\ua{b} \cap L) \subseteq\con.
\end{equation}
Clearly, $\conbar \cap (L \times L) = \con$, and $\conbar$ satisfies \eqref{C0}--\eqref{C3}.

For \eqref{C4}, suppose that $a,b,c \in M$ are such that $a \notconbar b$ and $a \notconbar c$. Then, there are $s_0,t_0,s_1,t_1\in L$ such that $a \leq s_0, b \leq t_0$ and $s_0\notcon t_0$, and $a \leq s_1, c \leq t_1$ and $s_1\notcon t_1$. It follows that $a \leq s_0 \cdot s_1$ and $b+c \leq t_0 + t_1$. But, by \eqref{C2} and \eqref{C3} for $\con$, we obtain that $s_0 \cdot s_1\notcon t_0$ and $s_0 \cdot s_1\notcon t_1$, therefore by \eqref{C4}, $s_0 \cdot s_1\notcon t_0+t_1$. In consequence,
\[
(\ua{a} \cap L) \times (\ua{(b+c)} \cap L) \nsubseteq\con
\]
which means $a\notconbar b+c$, as required.

Suppose that $\con'$ is a contact relation on $M$ which extends $\con$. Let $a,b \in M$ and $a \mathrel{\con'} b$. If $s,t \in L$ and $a \leq s, b \leq t$, then, $s \mathrel{\con'} t$ by \eqref{C3}. Since $\con'$ extends $\con$ we have $s \con t$, and therefore, $a \conbar b$.
\end{proof}
If $L$ is a p-algebra, then a contact relation on $\bL$ can be extended to $L$, even though $\bL$ need not be a subalgebra of $L$:
\begin{theorem}
Let $L$ be a p-algebra and $\con$ be a contact relation on $\bL$. Then, the relation $\conbar$ defined by
\begin{gather*}
a \conbar b \iffdef a^{**} \con b^{**}
\end{gather*}
is a contact relation on $L$.
\end{theorem}
\begin{proof}
    The only not completely trivial part is to show that $\conbar$ satisfies  \eqref{C4}. To this end, suppose that $a\conbar (b+c)$, i.e., $a^{**}\mathrel{\con} (b+c)^{**}$.
    Since
\begin{gather*}
 b^{**} \lor c^{**} = (b^{**} + c^{**})^{**} = (b+c)^{**},
\end{gather*}
the latter equation by e.g. \cite[Theorem VIII.2.1(ix)]{bd74}, we have that
\[
a^{**}\mathrel{\con} (b^{**}\lor c^{**})
\]
and so $a^{**}\con b^{**}$ or $a^{**}\con c^{**}$, by the assumption that $\con$ is a contact relation on $\bL$. In consequence, $a\conbar b$ or $a\conbar c$, as required.
\end{proof}
It was shown in \cite{Duntsch-Winter-CBCA} that there may be more than one contact relation on $M$ whose restriction to $L$ is $\con$.

\section{Standard topological representation}\label{sec:toprep}

If $\klam{X,\tau}$ is a topological space, we denote by $\RC(X)$ its Boolean algebra of regular closed sets. In this algebra, the join of two elements is their union and the meet is not intersection, but the closure of the interior of their intersection. The \emph{standard topological contact} is defined on $2^X$  by
\[
u \mathrel{\con_\tau} z \qtiff u \cap z \neq \z.
\]
It turned out that embedding a Boolean algebra with a contact relation into its spectral space of ultrafilters with standard contact was not in general possible, and other ways had to be found. One such way was to change the points from ultrafilters (which serve as points only in the case of the overlap---the minimal contact relation on any algebra) to a different kind. The solution was to use clans, a generalization of ultrafilters,  whose origins lay in proximity structures \cite{Thron-PSAG}. If $\klam{L,\con}$ is a  DLC, then a \emph{clan} $\Gamma$ is a union of prime filters of $L$ with the property that $\Gamma \times \Gamma \subseteq \con$; in particular, each prime filter is a clan.

We denote the set of clans of $\klam{L,\con}$ by $\Clan(L)$. Strictly speaking we should write $\Clan_{\con}(L)$; since there will be no opportunity for confusion, we shall use the simplified notation.  The close connection of $\con$ and $\Clan(L)$ is demonstrated by the following observation:

\begin{lemma}[{\citealp[Lemma 12]{Duntsch-et-al-DCLTR}}]\label{lem:CG}
For all $a,b \in L$, $a \con b$ \tiff there are prime filters $F$ and $G$ such that $\klam{a,b}\in F\times G\subseteq\con$  \tiff there exists a~clan $\Gamma$ such that $a,b \in \Gamma$.
\end{lemma}

It turned out, that the choice of clans as the set of points is appropriate for contact relations on Boolean algebras:
\begin{theorem}[{\citealp[Theorem 5.1]{Dimov-et-al-CARBTSPA1}}]\label{thm:standard}
Let $\klam{B,\con}$ be a Boolean algebra with a contact relation $\con$, $h\colon B \to 2^{\Clan(B)}$ be the embedding $h(a) \defeq \set{\Gamma \in \Clan(B): a \in \Gamma}$, and $\klam{X,\tau}$ the topological space with $X \defeq \Clan(B)$, and $\tau$ generated by the basis $\set{h(a): a \in B}$ for closed sets. Then, $\klam{X,\tau}$ is a compact semiregular $T_0$ space, $h$~is a dense embedding into the complete Boolean algebra $\RC(X)$ of regular closed sets of $X$, and $a \con b$ \tiff $h(a) \con_\tau h(b)$.
\end{theorem}
Similar to the embedding of Theorem \ref{thm:repDL}, the mapping $h$ sends each $a \in B$ to the set of points  containing $a$, however, unlike in Theorem \ref{thm:repDL}, the sets $h(a)$ are closed. Salient points of this representation are
\begin{align}
&\text{$X$ is the set of clans of $B$.} \label{s0}\\
&\text{$h$ embeds $B$ into the collection of regular closed sets as a basis for $\tau$.} \label{s2}\\
&\text{The contact relation on $2^X$ is the standard topological contact.} \label{s3}
\end{align}

In this section we shall investigate if and how these properties can be kept when considering lattice based  structures more general than Boolean algebras. Our first observation shows that keeping \eqref{s2} excludes large classes beyond that of Boolean algebras. To this end we define an \emph{underlap} relation $\underl$, dual to the overlap relation $\overl$ introduced in \eqref{df:overl}  by
\begin{gather}\tag{$\defeqtt{\underl}$}\label{def:ul}
a \mathrel{\underl} b \iffdef a + b \neq \one.
\end{gather}
In the case $a\underl b$ we say that $a$ \emph{underlaps}~$b$.
Note that $\underl$ depends only on the reduct $\klam{L,+,\one}$ and does not require a contact relation on $L$.

\begin{theorem}[{\citealp[Corollary 1]{Duntsch-et-al-DCLTR}}]
Let $\klam{X,\tau}$ be a topological space and $h$ be an embedding of the reduct $\klam{L,+,\zero,\one}$ into the lattice of closed sets such that $M \defeq \set{h(a): a \in L}$ is a closed basis for $\tau$.
The following are equivalent:
\begin{enumerate}
\item $\underl$ is extensional.
\item $M \subseteq \RC(X)$.
\item $(\forall a,b)[h(a \cdot b) = \Cl(\Intr(h(a) \cap h(b)))]$.
\end{enumerate}
\end{theorem}
This shows that the embedding even of a reduct of a distributive lattice into the lattice of regular closed sets---with or without a contact relation---is rather limited. Indeed, $\underl$-extensionality is not satisfied in $p$-algebras which are not Boolean, in particular, in Heyting algebras \cite[Theorem 1]{Duntsch-et-al-DCLTR}. It follows that we must give up semiregularity or other properties of the standard representation, when seeking to represent such algebras. To put it differently, a topological representation of any class of reducts of $p$-algebras cannot take place in the BA of regular closed subsets of a space. This is not desirable for a number of reasons. The motivation to talk about structures richer than Boolean algebras (in the sense of adding contact) stemmed from the idea of building algebraic models of space in which regions are taken as primitives.
We want regions to model (at least remotely) entities occupying the real space where physical phenomena take place. It was, quite naturally, assumed that regular closed subsets of the Cartesian three-dimensional space may serve this purpose well: they have no ``punctures'' or ``hairs'', and they have the well-known structure of a complete Boolean algebra \cite{Stell-BCAANATRCC,Pratt-Hartmann-EARIRBTOS}. Moreover, they allow for speaking about the nearness of regions in terms of the topological closure operation. Pure Boolean algebras are very rough in this respect, as the only nearness we can speak about in their case is the standard overlap operation, a very crude form of it (which is not surprising when we think about elements of BAs as clopen sets, regions ``deprived'' of boundaries). Therefore, the idea arose to expand the signature of BAs with a binary contact relation, whose aim was to grasp the phenomenon of nearness beyond the overlap relation. Hence, as well, the desire followed to have them represented as subalgebras of regular closed subsets of topological spaces. This desire  found its fulfillment in the representation theorems from \cite{Duntsch-et-al-RTBCA,Dimov-et-al-CARBTSPA1}. Going beyond $\RC(X)$ of some space $X$ may be considered as a form of philosophical betrayal, which we would very much want to avoid. Nevertheless, we venture at least to glimpse in such an unfavourable direction.

Let us first relax condition \eqref{s2}, so that $h(a)$ is a closed set that is not necessarily regular. In this case we have a positive result for a reduct of $L$:
\begin{theorem}\label{thm:repDCL}
Suppose $\klam{L,\con}$ is a DLC, and that $h\colon L \to 2^{\Clan(L)}$ is defined by the assignment $h(a) \defeq \set{\Gamma \in \Clan(L): a \in \Gamma}$. Then, $h$ is a $\klam{L,+,\zero, \one}$ embedding and $a \con b \Iff h(a) \cap h(b) \neq \z$.
\end{theorem}
\begin{proof}
Clearly, $h$ preserves $\zero$ and $\one$. If $a,b \in L$, then
\begin{xalignat*}{2}
h(a+b) &= \set{\Gamma \in \Clan(L): a + b \in \Gamma} \\
&= \set{\Gamma \in \Clan(L): a \in \Gamma \tor b \in \Gamma}, && \text{since $\Gamma$ is a clan} \\
& = h(a) \cup h(b).
\end{xalignat*}
This shows that $h$ preserves joins, and so, $\set{h(a): a \in L}$ can be used as a closed base for $\tau$. Furthermore, $h$ is injective, since there is a prime ideal $F$ containing exactly one of $a$ or $b \in F$ by the Prime Ideal Theorem. It remains to show that $h$ preserves contact. We have
\begin{gather}\label{hcon}
a \con b \Iff (\exists \Gamma \in \Clan(L))[a,b \in \Gamma] \Iff h(a) \cap h(b) \neq \z,
\end{gather}
the first equivalence by Lemma \ref{lem:CG}.
\end{proof}
Since $h$ preserves joins, the collection $\set{h(a): a \in L}$ may be regarded as a basis for a topology $\tau$ on $\Clan(L)$ with standard topological contact by \eqref{hcon}. However, the connection to the topological properties is secondary, and it seems not quite justified to speak of a topological representation.

Now, meet in the lattice of closed sets is intersection, and we require $h(a \cdot b) = h(a) \cap h(b)$. However, this may fail as the next examples show.
\begin{example}\label{ex:notmeet}
Suppose that $\zero$ is not meet irreducible, and let $\con$ be the largest contact relation on $L$, that is, $\con\defeq L^+ \times L^+$. Then, $L^+$ is the largest clan which contains all non-zero elements and it is an element of all $h(a), a \in L^+$. However, by the assumption, there are non-zero elements $a$ and $b$ such that $a\cdot b=\zero$, yet $h(a)\cap h(b)\neq\zero$.\QED
\end{example}

\pagebreak

Even when $\zero$ is meet irreducible, meet need not be preserved:
\begin{example}
Consider the following lattice:
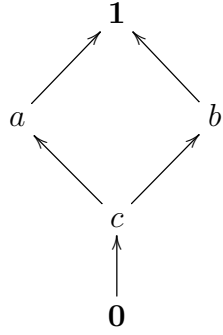
\begin{gather*}
\xymatrix{
& \one  &  & \\
a \ar@{->}[ru] & & b \ar@{->}[lu] \\
 & c\ar@{->}[ru] \ar@{->}[lu] \\
 & \zero \ar@{->}[u]
}
\end{gather*}
Again, let $\mathord{\con}\defeq L^+\times L^+$. Then, $\Clan(L) = \set{\ua{c}, \ua{a}, \ua{b}, \ua{a} \cup \ua{b}}$, and
\begin{gather*}
h(a \cdot b) = h(c) = \set{\ua{c}} \subsetneq \set{\ua{c}, \ua{a} \cup \ua{b}} = h(a) \cap h(b).
\end{gather*}\QED
\end{example}
 On the other hand, there are cases when meet is preserved:
\begin{example}
Suppose that $L$ is a linear order. Then, the only contact relation is $\overl$ and the clans are the end segments of $L$.  Let $h\colon L \to 2^{\Clan(L)}$ be the Stone map. By Theorem \ref{thm:repDCL} all that is left to show is that $h$ preserves meet. If $a,b \in L^+$, and \wlg $a \leq b$, then $h(a \cdot b) = h(a) = h(a) \cap h(b)$.\QED
\end{example}
Therefore, the best general result we can expect in spirit and form of Theorem \ref{thm:standard} by relaxing \eqref{s2} is a representation for bounded join reducts of bounded distributive lattices in the lattice of closed sets as in Theorem \ref{thm:repDCL}. The question remains whether disregarding regular closed sets is a meaningful action when we have spatial regions with boundaries in mind.

\section{Discrete representation}\label{sec:discrep}

An alternative to the standard topological representation of lattice based contact structures is a representation by relations on the set of prime filters. \citet{Duntsch-et-al-RBTODSAPA} developed the theory of discrete proximity Boolean algebras\footnote{For an explanation of the name ``discrete'' for these structures see \cite[Section 1]{Duntsch-et-al-RBTODSAPA}.} following, in part, Galton's theory of adjacency relations \cite{Galton-QSC}, and they related properties of proximities to properties of binary relations on the set of ultrafilters of a Boolean algebra $B$. Their work was extended in \cite{dw_cl}, where it is shown that there is a bijective order preserving correspondence between the contact relations on $B$ and the reflexive and symmetric relations on $\Prim(B)$ which are closed in the product topology of $\Prim(B) \times \Prim(B)$ of the spectral space of $B$ as in Theorem \ref{thm:repDL}. In this section we shall show that this result can be extended to DLCs. As a preparation we generalize \cite[Lemma 6]{Duntsch-et-al-RBTODSAPA}. For a filter $F$ of a DLC $L$ let $I_F \defeq \set{a \in L: (\exists b \in F)\,a \notcon  b}$.
\begin{lemma}\label{lem:FI}
Let $F,G$ be filters of $L$.
\begin{enumerate}
\item $I_F$ is an ideal of $L$.
\item $F \times G \subseteq \con$ \tiff $I_G \cap F = \z$ \tiff $I_F \cap G = \z$.
\item If $F \times G \subseteq \con$, there are prime filters $F',G'$ such that $F \subseteq F', G \subseteq G'$ and $F' \times G' \subseteq \con$.
\end{enumerate}
\end{lemma}
\begin{proof}
1. If $a \in I_L$, and $b \leq a$, then $b \in I_L$ by \eqref{C3}. If $a,b \in I_L$ with $c,d \in F$ are such that $a\notcon c$ and $b(-\con d)$, then $c \cdot d \in F$, and $c \cdot d \notcon a $ and $c \cdot d \notcon b $. \eqref{C4} now implies $c \cdot d\notcon (a+b)$.

2. \aright Assume that $a \in I_F \cap G$, and let $b \in F$ such that $a \notcon  b$. Since $F \times G \subseteq \con$, we have $a \con b$, a contradiction.

\aleft Let $a \in F, b\in G$, and $a (-\con b)$. Then, $a \in F \cap I_G \neq \z$. The remaining implications follow from the symmetry of $\con$.

3. Let $F \times G \subseteq \con$; then $I_F \cap G = \z$ by 1. above. By the prime ideal theorem, there is some prime filter $G'$ with $G \subseteq G'$ and $G' \cap I_F = \z$. By 2. above, this shows that $F \times G' \subseteq \con$. By symmetry of $\con$  and 2. there is some prime filter $F'$ containing $F$ and $F' \times G' \subseteq \con$.
\end{proof}
Let $\klam{X,\tau}$ be the Stone space of $L$, $h: L \into 2^X$ be the embedding of Theorem \ref{thm:repDL}, and set $\Base\defeq \set{h(a): a \in L}$; then, $\Base$ is an open basis for $\tau$, and each element of $\Base$ is compact. Furthermore, $\Base \times \Base$ is an open basis for the product topology on $X \times X$ \cite[Proposition 2.3.1]{eng77}. Let $\rsc$ be the set of reflexive and symmetric binary relations on $X$ that are closed in the product topology $X \times X$ ordered by set inclusion. We define mappings $p\colon \ConRel \to \rsc$ and $q\colon \rsc \to \ConRel$ as follows:
\begin{align}\label{def:p}
\con &\overset{p}{\longmapsto} \set{\klam{F,G}: F \times G \subseteq \con}, \\
R &\overset{q}{\longmapsto} \bigcup\set{F \times G: \klam{F,G} \in R}. \label{def:q}
\end{align}

\begin{lemma}\label{lem:for-dr}
    Let $\klam{X,\tau}$ be the Stone space of $L$. If $R$ is a closed relation in $X\times X$, then for any prime filters $F$ and $G$ of $L$
    \[
F\times G\subseteq q(R)\qtiff\klam{F,G}\in R.
    \]
\end{lemma}
\begin{proof}
The non-trivial direction is left-to-right. Assume that $F\times G\subseteq q(R)$ and let $h(a)\times h(b)$ be an arbitrary basic open set around $\klam{F,G}$. Then, $\klam{a,b}$ is an element of $q(R)$, so there are prime filters $F'$ and $G'$ such that $\klam{F',G'}\in R$ and $\klam{a,b}\in F'\times G'$, i.e., $\klam{F',G'}\in h(a)\times h(b)$. In consequence, $[h(a)\times h(b)]\cap R\neq\emptyset$, and since $h(a)\times h(b)$ is arbitrary and $R$ is closed, $\klam{F,G}\in R$, as required.
\end{proof}

\begin{theorem}\label{thm:dr}
\begin{enumerate}
\item The mappings $p$ and $q$ are well defined, i.e. $p(\con) \in \rsc$ and $q(R) \in \ConRel$.
\item $p$ and $q$ are bijective and inverses of each other.\qed
\end{enumerate}
\end{theorem}
\begin{proof}
1. Let $\con \in \ConRel$. Then, $p(\con)$ is reflexive and symmetric by \eqref{C1} and \eqref{C2}. Suppose that $\klam{F,G} \in \Cl(p(\con))$, and assume that $\klam{F,G} \not\in p(\con)$.  By definition of $p$ we have $F \times G \not\subseteq \con$, and so there are $a \in F, b \in G$ such that $a \notcon b$. Since $h(a) \times h(b)$ is an open neighbourhood of $\klam{F,G}$ and $\klam{F,G} \in \Cl(p(\con))$, there are $F',G' \in \PrimF(L)$ such that $a \in F', b \in G'$ and $\klam{F',G'} \in p(\con)$. It follows from the definition of $p$ that $a \con b$, contradicting our assumption.

Conversely, let $R \in \rsc$. We show that $q(R)$ is a contact relation on $L$.

\eqref{C0}: $\zero$ is not in contact to any $x \in L$, since $\zero \not\in F$ for all $F \in \PrimF(L)$.

\eqref{C1}:  If $a \in L^+$, there is some $F \in \PrimF(L)$ such that $a \in F$. The reflexivity of $R$ implies $\klam{F,F} \in R$, and $a \in F$ and $F \times F \subseteq q(R)$ imply $a \mathrel{q(R)} a$.

\eqref{C2}: This follows from the symmetry of $R$.

\eqref{C3}: Let  $a \mathrel{q(R)} b$, and $b \leq c$. If $a \in F, b \in G$, then $c \in G$ since $G$ is a filter.

\eqref{C4}: Let $a \mathrel{q(R)} (b+c)$, $a \in F, b+c \in G$ and $F \mathrel{R} G$. Since $G$ is prime, we have $b \in G$ or $c \in G$, and therefore, $a \mathrel{q(R)} b$ or $ \mathrel{q(R)} c$.

Note that the closedness of $R$ was not required in this part of the proof.

\smallskip

2. First, we will show that $q(p(\con)) = \con$ for $\con \in \ConRel$. This implies that $p$ is injective and $q$ is surjective.

\sbe Let $a \mathrel{q(p(\con))} b$. Then, there are $F,G \in \PrimF(L)$ such that $\klam{F,G} \in p(\con)$ and $a \in F, b \in G$. Since $\klam{F,G} \in p(\con)$ \tiff $F \times G \subseteq \con$, we obtain that $a \mathrel{\con} b$.

\spe Let $a \con b$, and consider the principal filters $\ua{a}$ and $\ua{b}$; then, $\ua{a} \times \ua{b} \subseteq \con$ by the properties of $\con$. By Lemma \ref{lem:FI}(3) there are prime filters $F,G$ such that $a \in F, b \in G$ and $F \times G \subseteq \con$. Then, $\klam{F,G} \in p(\con)$, and $a \mathrel{q(p(\con))} b$ by the definition of $q$.

It remains to show that $q$ is injective, since then $q$ is bijective and so is $p$ as the inverse of $q$. Let $R,S \in \rsc$ be such that $q(R) = q(S)$, and assume that $R \neq S$, \wlg that $\klam{F,G} \in R \setminus S$. Since $R$ is closed, there are $a,b \in L$ such that $\klam{F,G} \in h(a) \times h(b)$ and $[h(a) \times h(b)] \cap S = \z$. Since $\klam{F,G} \in R$ we have $F \times G \subseteq q(R)$, and $q(R) = q(S)$ implies that $F \times G \subseteq q(S)$. By Lemma \ref{lem:for-dr}, $\klam{F,G}\in S$. Now, $\klam{F,G} \in h(a) \times h(b)$ contradicts $[h(a) \times h(b)] \cap S = \z$.
\end{proof}

The representation theorem for bounded distributive lattices entails representations for other classes of algebras. For example, if $\klam{L,+,\cdot,-,\zero,\one}$ is a Boolean algebra, then its reduct $L_r$ to the signature $\{+,\cdot,\zero,\one\}$ is a bounded distributive lattice. The notion of prime filter, and the embedding $h\colon L_r\to 2^{\Prim(L_r)}$ are the same, thus the points and the basis for the Stone space of $L_r$ are the same as those for the Stone space of $L$. In consequence, we have identical spaces and the same closed symmetric, reflexive, and closed relations on their product space. In consequence, the correspondence between $\rsc$ of $L_r$ and the lattice of contacts of $L_r$ is analogous to the correspondence between $\rsc$ of $L$ and the lattice of contacts of $L$. Similar observations let us deduce the same correspondence for all algebras that are expansions (in the sense of the signatures) of bounded distributive lattices.

While Theorem \ref{thm:dr} shows that every contact relation is of the form $q(R)$ for some symmetric and reflexive relation on $\Prim(L)$, the closure of $R$ is essential, and not the properties of the lattice. There are relations on $\Prim(L)$ which are not of the form $p(\con)$ for any contact relation on $L$, even when $L$ is a Boolean algebra as the following example from  \cite{Duntsch-et-al-RBTODSAPA} shows.

\begin{example}
Suppose that $B$ is the finite-cofinite algebra of $\omega$, and define $R \defeq \set{\klam{\ua{a},\ua{b}}: a,b \in\Atom(B)}$, where $\Atom(B)$ is the set of atoms of $B$. Then, $R$ is not the universal relation, since the cofinite $F$ filter is not $R$-related to anything, in particular $\klam{F,F}\notin R$. But for any $x,y\in F$, $[h(x)\times h(y)]\cap R\neq\emptyset$. Indeed, for any pair of atoms $a\leq x$ and $b\leq y$ we have $\klam{\upop a,\upop b}\in R$. Thus $\klam{F,F}\in\Cl R$ and in consequence, $R$ is not closed. But it is reflexive and symmetric. Assume that there is some $\con\in \ConRel$ such that $R = p(\con)$. Let $a,b \in B^+$, and choose $n \in a, m \in b$. Then, $\ua{\{n\}}, \ua{\{m\}}$ are principal ultrafilters and $R = p(\con)$ implies $\ua{\{n\}} \times \ua{\{m\}} \subseteq \con$. Now, \eqref{C3} implies that $a \con b$, and therefore, $\con$ is the universal relation on $B^+$. This implies $p(\con) = \Prim(B) \times \Prim(B) \neq R$.
\QED\end{example}

For both mappings, $p$ and $q$, we have that they are isotone:
\begin{align*}
    R_1\subseteq R_2&{}\rarrow q(R_1)\subseteq q(R_2),\\
    \con_1\subseteq\con_2&{}\rarrow p(\con_1)\subseteq p(\con_2)\,.
\end{align*}

    Let $\rs$ be the family of all reflexive and symmetric relations on $\Prim(L)$. As we have seen in the proof of Theorem~\ref{thm:dr}, there is a map $\bar{q}\colon\rs\to\ConRel$ (an expansion of $q$ to the family of all reflexive and symmetric relations) which sends every reflexive and symmetric relation $R$ to the contact relation $\bar{q}(R)$. This mapping is isotone, too, and it is onto. Since $p\colon\ConRel\to\rsc$ is a bijection, $p\circ\bar{q}\colon\rs\to\rsc$ is a~surjection. We will show that $p\circ\bar{q}$ is the topological closure operator restricted to $\rs$. Let $c\defeq p\circ\bar{q}$.

    Firstly, let us show that $R\subseteq c(R)$. Indeed, if $\klam{F,G}\in R$, then $F\times G\subseteq\bar{q}(R)$, and so $\klam{F,G}\in c(R)$. Secondly, let us suppose that $Q\in\rsc$ is such that $R\subseteq Q$. Then $\bar{q}(R)\subseteq q(Q)$, and so $p(\bar{q}(R))\subseteq p(q(Q))$, i.e., $c(R)\subseteq Q$, since $Q$ is closed.

    What is interesting here is that the closure of $R\in\rs$ is obtained via taking the contact relation on the algebra. In consequence, for every relation $R$ on the product space, we may find the unique contact relation $\con$ on the algebra corresponding to $R$: take $R^{\star}$ to be the reflexive and symmetric closure of $R$, and then take $\bar{q}(R^{\star})$. Further, if we look at the family of all relations $\frR$ on the product space and its quotient $\frR/_{\sim}$ w.r.t.
    \[
        R\sim Q\iffdef \bar{q}(R^{\star})=\bar{q}(Q^{\star})
    \]
    there exists a bijective correspondence $\alpha$ between $\frR/_{\sim}$ and $\ConRel$ given by $\alpha(R/_{\sim})\defeq \bar{q}(R^{\star})$.

\subsection{Discrete representation of distributive join-semilattices with \texorpdfstring{$\zero$}{0}}

As it turns out, the results we obtained can be extended to structures with reduced signatures. In this section, we will suppose that $\klam{L,+,\zero}$ is a distributive join-semilattice with smallest element $\zero$ and with at least two elements. These were introduced by \citet{gs1962}. Due to the form of the signature, the distributivity property is defined in the following way:
\[
a\leq x+y\rarrow(\exists x_1,y_1\in L)\,(x_1\leq x\tand y_1\leq y\tand a = x_1+y_1).
\]

An \emph{ideal} of a join semi-lattice $L$ is defined in the standard way. $F$ is \emph{filter} of a join semi-lattice $L$ if it satisfies the following two conditions:
\begin{enumerate}[label=(\roman*)]
    \item for any $a,b\in F$ there is an $x\in F$ which is a lower bound of $a$ and $b$,
    \item if $a\in F$ and $a\leq x$, then $x\in F$.
\end{enumerate}
$F$ is a \emph{prime} filter, if for any $a,b\in L$, $a+b\in F$ implies that $a\in F$ or $b\in F$.

An ideal $I$ is \emph{prime} if $L\setminus I$ is a filter.

\begin{lemma}
    For any ideal $I$, $I$ is a prime ideal iff $F\defeq L\setminus I$ is a prime filter.
\end{lemma}
\begin{proof}
    ($\Rightarrow$) $F$ is a filter by the assumption and by the definition of a prime ideal. If $a+b\in F$, then $a+b\notin I$, so one of $a$ and $b$ is not in the $I$. Thus, $a\in F$ or $b\in F$, and $F$ is prime.

    \smallskip

    ($\Leftarrow$) By the definition of a prime ideal.
\end{proof}
Let $\PrimF(L)$ and $\PrimI(L)$ be, respectively, the sets of all prime filters and prime ideals of $L$. The Stone space of $L$ may be constructed in two equivalent ways: either by defining $h\colon L \to 2^{\PrimF(L)}$ by $x\mapsto\{F\in\PrimF(L):x\in F\}$, or---as it is done in \cite{gra78}---by $g\colon L \to 2^{\PrimI(L)}$ by $x \mapsto \set{P \in \PrimI(L): x \not\in P}$. In both cases, we impose a topology on the set of points by taking as a basis of open sets either $\{h(x):x\in L\}$ or $\{g(x):x\in L\}$. To be consistent with the previous sections, we choose the method via prime filters. Let $\S(L)\defeq\klam{\Prim(L),\tau}$ be the Stone space of $L$.

Suppose that $\con$ is a contact relation on $L^+$. For a filter $F$ of $L$ let
\[
I_F \defeq \set{a \in L: (\exists b \in F)\,a \notcon  b}.
\]
With this, the counterpart of Lemma~\ref{lem:FI} holds for distributive join-semilattices. Further, for the family $\rsc$ of closed, reflexive, and symmetric relations in $\S(L)\times\S(L)$ we define the mapping $q\colon\rsc\to 2^{L\times L}$, and for the lattice of contact relations $\ConRel$ on $L$, the mapping $p\colon\ConRel\to2^{\S(L)\times\S(L)}$. For these we can show the counterpart of Theorem~\ref{thm:dr}. In consequence, we obtain the following
\begin{theorem}\label{th:max-gen-of-Th-1-DW}
    If $L$ is a distributive join-semilattice with zero, then the family of contact relations is in one-to-one correspondence with the family of closed, reflexive, and symmetric relations on the binary product of the Stone space of $L$.
\end{theorem}
The above theorem is the direct strengthening of \cite[Theorem 1]{dw_cl}, and has probably the most general form it can take, taking into account the fact that the standard axioms for the contact relation are formulated in the signature of join-semilattices with the bottom element.

\label{page:Ivanova}
Related results can be found in \cite{Duntsch-et-al-PRODBPL-arxiv} and \cite{Ivanova-CJS}. In the former paper, the authors develop a~representation (and duality) for bounded distributive lattices with the contact (and also the pre-contact) relation in the manner of Priestley. The results of the latter concern bounded distributive join-semilattices, and the contact relations considered, besides the standard axioms \eqref{C0}--\eqref{C4}, satisfy also infinitely many axiom schemas expressing additional properties of contact. In addition, the author proves the representation of these structures in the spirit of \cite{Duntsch-et-al-RTBCA,Dimov-et-al-CARBTSPA1}. Both approaches are different from the discrete approach followed in our work.

\section{The structure of \texorpdfstring{$\ConRel$}{C}}\label{sec:conlat}
Let us investigate the structure of $\ConRel$, respectively $\rsc$, using the notation of the previous section. It is well known \cite{mkt44,mkt46} that the closed sets of any topological space $X$ form a complete co-Heyting algebra $H(X)$  under the operations
\begin{gather}\label{clcH}
\bigvee A \defeq \Cl\left(\bigcup A\right),\ \bigwedge A\defeq \bigcap A,\ a \leftarrow b\defeq \Cl(a \setminus b),\  \zero\defeq \z,\ \one\defeq X.
\end{gather}
Our aim is  to show that $\rsc$ is a complete sublattice of $H(X)$, albeit with a different $\zero$, generalizing a corresponding result in \cite{dw_cl} in which $X$ is the Stone space of a Boolean algebra.

Keeping in mind that $p[\ConRel] = \rsc$, and that the smallest contact relation on $L$ is the overlap relation $\overl$ we shall first look at $p(\overl)$. If $L$ is a Boolean algebra, then $p(\overl)$ is the identity on $\Prim(L)$ which is closed, since the space is Hausdorff \cite[Exercise 2.3.C.(a)]{eng77}, but it will be larger when it is not.

\pagebreak

\begin{lemma}\label{lem:poverl}
$p(\overl) = \set{\klam{F,G}: F \tand G \text{ have an upper bound in } \Prim(L)}$, and $p(\overl) = \min(\rsc)$.
\end{lemma}
\begin{proof}
By the definition of $p$,
\begin{align*}
\klam{F,G} \in p(\overl) &\Iff F \times G \subseteq \overl, \\
&\Iff a \in F \tand b \in G \Implies a \cdot b \neq \zero, \\
&\Iff \text{the filter $H$ generated  by $F \cup G$ is a proper filter}, \\
&\Iff F \tand G \text{ have an upper bound in } \Prim(L).
\end{align*}
Clearly, the smallest element of $\rsc$ is the closure of the identity relation $1'$, so it suffices to show that $p(\overl) \subseteq \Cl(1')$. Suppose that $\klam{F,G} \in p(\overl)$, and let $h(a) \times h(b)$ be a basic neighbourhood of $\klam{F,G}$. By the hypothesis, $a \cdot b \neq \zero$, so there is a prime filter $H$ containing both $a$ and $b$. Then, $\klam{H,H} \in [h(a) \times h(b)] \cap 1'$.
\end{proof}
Observe that $F$ and $G$ have an upper bound in $\Prim(G)$ \tiff $F \cup G$ has the finite intersection property.

\begin{corollary}\label{cor:max-from-trans}
    If $R\in \rsc$ is transitive, then for any $\klam{F,G}\in R$ there exists a maximal $F'\supseteq F$ such that $\klam{F',G}\in R$.
\end{corollary}
\begin{proof}
    Let $\klam{F,G}\in R$. Since for a maximal prime filter $F'$ extending $F$, $F'\cup F$ clearly have a finite intersection property, the pair $\klam{F',F}$ is in $p(\overl)$, and also in $R$, by Lemma~\ref{lem:poverl}. So, the transitivity of $R$ entails that $\klam{F',G}\in R$.
\end{proof}

Our next step is to show that the lattice operations of \eqref{clcH} preserve reflexivity and symmetry. Clearly, reflexivity is preserved by supersets, and symmetry is preserved under arbitrary unions. It remains to show that symmetry is preserved under closure.
\begin{lemma}\label{lem:sym}
Let $R$ be a symmetric binary relation on $\Prim(L)$. Then, $\Cl(R)$ is in $\rsc$.
\end{lemma}
\begin{proof}
Let $\klam{F,G} \in \Cl(R)$. Then, every basic neighbourhood $h(a) \times h(b)$ intersects $R$, and it follows from the symmetry of $R$, that every basic neighbourhood $h(b) \times h(a)$ of $\klam{G,F}$ also intersects $R$. This implies that $\klam{G,F} \in \Cl(R)$.
\end{proof}

\begin{theorem}\label{thm:rscrep}
$\rsc$ is a complete atomic co-Heyting algebra with supremum, infimum and $\one$ as in \eqref{clcH}, $\zero = \Cl(1')$, and $R \leftarrow S=\Cl(R \setminus S) \cup \Cl(1')$.
\end{theorem}
\begin{proof}
All that is left to show is that $R \leftarrow S$ is the dual relative complement in $\rsc$, and to exhibit the atoms of $\rsc$. $\Cl(R \setminus S)$ is symmetric by Lemma \ref{lem:sym}, and it is the smallest closed relation $T$ on $\Prim(L)$ with $R \subseteq S \cup T$ by  \eqref{clcH}. Since $\Cl(1')$ is the smallest closed, reflexive, and symmetric relation, $R \leftarrow S$ is the smallest closed, reflexive, and symmetric $T$ for which $R \subseteq S \cup T$.

The atoms of $\rsc$ have the form $p(\overl) \cup (\da{F} \times \da{G}) \cup (\da{G} \times \da{F})$. This follows from the fact that the minimal closed nonzero subsets of $\Prim(L)$ have the form $\da{F}$ for $F \in \Prim(L)$ and that the closure of a product of two nonempty sets is the product of the closures.
\end{proof}

\pagebreak

Using Theorem \ref{thm:dr} we immediately obtain the algebraic structure of $\ConRel$:
\begin{theorem}\label{th:co-Heyting}
$\ConRel$ is a complete atomic co-Heyting algebra with smallest element~$\overl$, largest element $L^+ \times L^+$, and for $\frR \defeq \set{\con_i: i \in I}\subseteq\ConRel$
\begin{align*}
\bigvee \frR &= q\left(\bigvee\set{p(\con_i): i \in I}\right), \\
\bigwedge \frR &= q\left(\bigwedge\set{(p(\con_i): i \in I}\right), \\
\con \leftarrow \con' &= q(p(\con) \leftarrow p(\con')).
\end{align*}
\end{theorem}
Similarly as in the case of Theorem~\ref{th:max-gen-of-Th-1-DW} and Theorem 1 from \cite{dw_cl}, theorems~\ref{thm:rscrep} and \ref{th:co-Heyting}
generalize directly Theorem~3 and Corollary~1 from the aforementioned work.

$\ConRel$ is usually not closed under intersection, as a simple example in \cite[p. 105]{dw_cl} shows. However, if $\frR$ is a descending chain, then meet and intersection coincide.
Several axioms other than \eqref{C0}--\eqref{C4} were considered in the literature. The \emph{strong} or \emph{interpolation axiom} is as follows with ${}^\ast$ as (pseudo-)complement:\footnote{For a history of the strong axiom see, for example, \cite{Naimpally-et-al-PS}.}
\begin{gather}\tag{Str}\label{C6}
(\forall x,y)[x \notcon y \Implies (\exists z)(x\notcon z \tand y\notcon z^\ast)].
\end{gather}
In \cite[Theorem 10(3)]{Duntsch-et-al-RBTODSAPA} it was proved that in the case of Boolean algebras, \eqref{C6} is equivalent to transitivity of the corresponding relation on the space of ultrafilters. Below, we analyze the transitivity property for p-algebras, and we show that for them, the aforementioned equivalence to the strong axiom fails.

\begin{lemma}\label{lem:C(x)=C(x**)}
    If $L$ is a $p$-algebra, $R\in\rsc$ is transitive and $\con\defeq q(R)$, then for all $x,y\in L$
    \[
x\con y\iff x^{\ast\ast} \con y.
    \]
\end{lemma}
\begin{proof}
($\Rightarrow$) If $x\con y$, then $x^{\ast\ast}\con y$ by $x\leq x^{\ast\ast}$ and \eqref{C3}.

\smallskip

($\Leftarrow$) Suppose that $x^{\ast\ast} \con y$. Thus, by Lemma~\ref{lem:CG}, there are prime filters $F$ and $G$ such that $\klam{x^{\ast\ast},y}\in F\times G\subseteq\con$. Thus, $\klam{F,G}\in R$, and, by Corollary~\ref{cor:max-from-trans}, transitivity of $R$ implies that there must be
a~maximal filter $F'$ such that $x^{\ast\ast}\in F'$, and $F'\times\{y\}\subseteq\con$. Being maximal, $F'$ is prime, and by Lemma~\ref{lem:max}(3), $x\in F'$, so $x\con y$.
\end{proof}

Thus, as we can see, transitivity of $R\in\rsc$ entails that for $\con\defeq q(R)$ and for any region $x$ we have $\con(x)=\con(x^{\ast\ast})$. This is, however, a property stronger than the transitivity of $\con$ in general.  For example, in any Boolean algebra we obviously have $\con(x)=\con(x^{\ast\ast})$, but $p(\con)$ may not be a transitive relation.

One more observation follows from Lemma~\ref{lem:C(x)=C(x**)}.

\begin{corollary}\label{cor:ext-p-algebra-to-BA}
    If $\klam{L,\con}$ is a p-algebra, $R\in\rsc$ is transitive and $\con\defeq q(R)$ is extensional, then $L$ is a Boolean algebra.
\end{corollary}
\begin{proof}
    We have $x=x^{\ast\ast}$ for any $x\in L$ as a consequence of extensionality of $\con$. This suffices to conclude that $L$ is a Boolean algebra.
\end{proof}

\begin{theorem}
    If $L$ is a $p$-algebra, $R\in\rsc$ is transitive, and $\con\defeq q(R)$, then the interpolation axiom holds for $\con$.
\end{theorem}
\begin{proof}
    For the proof by contraposition, let us assume that $a$ and $b$ are such that $(\dagger)$ for every $c\in L$, $a\con c$ or $c^{\ast}\con b$. Let $I\defeq\{d\in L: a\notcon  d\}$ and $F\defeq\{d\in L\mid d^{\ast}\notcon  b\}$. By $(\dagger)$, $I\cap F=\emptyset$. Since $I$ is an ideal and $F$ is a filter, there is a prime filter $U\supseteq F$ that is disjoint from $I$. We have ($\ddagger$) $\upop a\times U\subseteq\con$. Indeed, if $y\in U$, then $y\notin I$ and so $a\con y$. Therefore, $\{a\}\times U\subseteq \con$, and by \eqref{C3} we obtain ($\ddagger$). Further, we also have  ($\maltese$) $U\times\upop b\subseteq\con$. Indeed, if $x\in U$ is such that $x\notcon  b$, then $x^{\ast\ast}\notcon  b$ by Lemma~\ref{lem:C(x)=C(x**)}. So $x^{\ast}\in F$ by the definition of $F$, and in consequence $x^{\ast}\in U$. But then, both $x$ and $x^{\ast}$ are in $U$, which is a contradiction.

    So, with $(\dagger)$ and $(\maltese)$, we expand $\upop a$ and $\upop b$ to prime filters, respectively, $H$ and $G$ such that $H\times U\subseteq\con$ and $U\times G\subseteq\con$. Thus, $\klam{H,U}\in R$ and $\klam{U,G}\in R$, and by the transitivity of $R$ we obtain $\klam{H,G}\in R$, i.e., $H\times G\subseteq\con$. This implies that $a\con b$, as required.
\end{proof}
The next example shows that the converse does not necessarily hold.
\begin{example}
Consider the following p-algebra $L$:
\begin{gather*}\label{fig:Stone}
\xymatrix{
& \one  &  & \\
& c \ar@{->}[u]&  \\
a \ar@{->}[ru] && \ar@{->}[lu] b \\
& \zero \ar@{->}[lu] \ar@{->}[ru] &
}
\end{gather*}
with the $\overl$ relation on $L^+$. Since $x \cdot y = \zero$ \tiff  $\zero \in \set{x,y}$ or $\set{x,y} = \set{a,b}$ it is easily seen that $\overl$ satisfies \eqref{C6}. The prime filters of $L$ are $\ua{a}, \ua{b}$, and $\set{\one}$. Now, $\ua{a} \times \set{\one} \subseteq \overl$ and $\set{\one} \times \ua{b} \subseteq \overl$, but $\ua{a} \times \ua{b} \not\subseteq \overl$. This shows that $p(\overl)$ is not transitive.
\QED\end{example}
\section{Summary and outlook}

We have investigated contact relations on bounded distributive lattices. In particular, we have argued that a  topological representation as for Boolean contact algebras is not feasible in the more general situation. Instead, generalizing results from \cite{Duntsch-et-al-RBTODSAPA} we have shown that a discrete representation relating contact relations on $L$ with closed reflexive and symmetric relations on the space of prime filters of $L$ is possible. Using this result we have shown that the contact relations on $L$ form a complete atomic co-Heyting algebra, thus extending a result from \cite{dw_cl}. Finally, we have exhibited a connection between a contact relation on a p-algebra satisfying the strong axiom \eqref{C6} and the transitivity of its counterpart relation on $\Prim(L)$. In future work we shall consider further axioms for contact relations on bounded distributive lattices which were investigated on Boolean contact algebras such as extensionality \eqref{C5} and connectivity, i.e.
\begin{gather}\tag{Connect}\label{C7}
x+y = \one \timplies x \con y.
\end{gather}
Considering \eqref{C6}, we note that it requires some form of complementation, and it is of interest to consider formulations on a DLC which reduce to \eqref{C6} on p-algebras. One such form is as follows: Since $u \cdot z = \zero \Iff u \leq z^*$, by \eqref{C3} we have
\begin{align}
y(-\con)z^* &\Iff (\forall u)[u \cdot z = \zero \timplies y(-\con)u], \\
y \con z^* &\Iff (\exists u)[u \cdot z = \zero \tand y \con u].
\end{align}
This way, \eqref{C6} becomes
\begin{gather}\tag{Str'}\label{C6'}
(\forall x,y)[x(-\con)y \Implies (\exists z)(x(-\con)z \tand \underbrace{(\forall u)(u \cdot z = \zero \timplies y(-\con)u)}_{y(-\con)z^*})].
\end{gather}
We also intend to investigate the effect of \eqref{C6'} on the corresponding prime filter relation of a DLC.

\section*{Acknowledgements}
This research was funded by the National Science Center (Poland), grant number 2020/39/B/HS1/00216, ``Logico-philosophical foundations of geometry and topology''. Ivo D{\"u}ntsch gratefully acknowledges the hospitality and support of the Department of Logic in the Institute of Philosophy, Nicolaus Copernicus University in Toru\'n.

\bibliographystyle{apalike}

\providecommand{\noop}[1]{}

\end{document}